\documentclass[letterpaper, 10pt]{amsart}

\usepackage{mymacros}

\begin{document}

\title{Orbifold Biquotients of $SU(3)$}
\author{Dmytro Yeroshkin}
\address{University of Pennsylvania, Department of Mathematics, Philadelphia, Pa 19104, USA}
\email{dmytroy@math.upenn.edu}

\begin{abstract}
One of the main methods of constructing new spaces with positive or almost positive curvature is the study of biquotients first studied in detail by Eschenburg. In this paper we classify orbifold biquotients of the Lie Group $SU(3)$, and construct a new example of a 5-dimensional orbifold with almost positive curvature. Furthermore, we extend the work of Florit and Ziller on the geometric properties of the orbifolds $SU(3)//T^2$.
\end{abstract}
\maketitle

\section{Introduction}

One of the fundamental questions in Riemannian geometry is what spaces admit metrics with positive sectional curvature. The only known compact simply connected manifolds with positive sectional curvature are the compact rank one symmetric spaces, and some examples in dimension below 25, most given as quotients of Lie groups with non-negative curvature (see \cite{ZExam} for an overview). One can also ask similar questions with slightly weaker conditions, such as almost-positive curvature, i.e. spaces with non-negative curvature and positive curvature almost everywhere; or spaces with quasi-positive curvature, i.e. non-negative curvature and at least one point with all planes having positive curvature.

Recall that biquotients have the form $K\backslash G/H$ where $H,K$ are subgroups of $G$. Alternatively, these quotients can be written as $G//U$ where $U\subset G\times G$ acts on both the left and the right. Starting with the work of Gromoll and Meyer (\cite{GMS7}) and Eschenburg (\cite{ENewEx}, \cite{EHab}), biquotients have proven to be a plentiful source of new examples of manifolds with positive, almost-positive and quasi-positive curvature (see for example \cite{WApos}, \cite{TQPC}, \cite{EKSph}, \cite{KThes}, \cite{DVThes} and \cite{KAPC}). While the focus has been primarily on manifolds, some work has also been done on orbifolds (see \cite{FZOrbi} and \cite{KThes}). Even in the study of manifolds, orbifolds can be used as a tool, for example, the positively curved cohomogeneity one manifold $P_2$ in \cite{GVZPosC} and \cite{DP2}, admits a bundle structure over an orbifold, as do the remaining cohomogeneity one positive curvature candidates, see \cite{GWZCoho}.

In his thesis, DeVito classified low dimensional smooth biquotients up to diffeomorphism \cite{DVThes}. Since the first interesting biquotients of positive curvature given by Eschenburg are of the form $SU(3)//S^1$, it is natural to search for further interesting biquotients of $SU(3)$. Furthermore, $SU(3)$ is the Lie group of lowest dimension whose orbifold biquotients have not been studied in detail (see \cite{KThes} and \cite{KCurv} for orbifold biquotients of $S^3\times S^3$). We point out that they have a natural metric of non-negative sectional curvature induced by the bi-invariant metric on $SU(3)$.

\begin{mthm}\label{ThmClass}
The following is a complete list of orbifold biquotients $SU(3)//U$ with $U$ connected:

\begin{enumerate}
\item Homogeneous spaces:

\subitem Classical manifolds $S^5=SU(3)/SU(2)$, and $\CP^2=SU(3)/U(2)$

\subitem Wallach spaces $W^7_{p,q}=SU(3)/S^1_{p,q}$, $W^6=SU(3)/T^2$

\subitem The Wu manifold $SU(3)/SO(3)$

\item Generalized Eschenburg spaces and orbifolds $SU(3)//S^1$ and $SU(3)//T^2$

\item Weighted projective spaces of the form $SU(3)//U(2)$ and $SU(3)//(SU(2)\times S^1)$

\item Circle quotients of the Wu manifold $S^1_{p,q}\backslash SU(3)/SO(3)$

\item One orbifold of the form $SU(3)//SU(2)$.
\end{enumerate}
\end{mthm}

The homogeneous spaces are well-known, and the positively curved ones fully classified (see \cite{WHomPos},\cite{BBHomPos}). The seven dimensional Eschenburg spaces were introduced in \cite{EHab}, and the six dimensional orbifolds in \cite{FZOrbi}. Recall that weighted projective spaces which appear in many contexts, see e.g. \cite{DWPV}, and are of interest to both algebraic and differential geometers, are defined as $\CP^2[\lambda_0,\lambda_1,\lambda_2] = S^5/S^1$ where the $S^1$-action is given by
\[
w\star(z_0,z_1,z_2) = (w^{\lambda_0} z_0, w^{\lambda_1} z_1, w^{\lambda_2} z_2),
\]
where $\lambda_i\in\bbZ\setminus\{0\}$, $\gcd(\lambda_0,\lambda_1,\lambda_2)=1$ and $(z_0,z_1,z_2)$ are coordinates on $\bbC^3\supset S^5$. We examine weighted projective spaces more closely in section \ref{SecPfClass}, in particular, we will show that every weighted $\CP^2$ can be obtained as a biquotient of $SU(3)$. The orbifold structure of circle quotients of the Wu manifold will be discussed in Section \ref{SecNew}. They give rise to an interesting class of 4-dimensional biquotients with non-negative curvature.

Of special interest is the new example $SU(3)//SU(2)$. It is obtained by embedding $\phi:SU(2)\hookrightarrow SU(3)\times SU(3)$ with $\phi(g)=(g,\pi(g))$, where $\pi$ is the 2-fold covering map $\pi:SU(2)\to SO(3)$. Here we are able to improve the natural metric of non-negative sectional curvature as follows:

\begin{mthm}\label{ThmO5}
The orbifold $\calO^5=SU(3)//SU(2)$ admits a metric with almost positive curvature such that

\begin{enumerate}
\item The set of points with 0-curvature planes forms a totally geodesic, flat 2-torus $T$ which is disjoint from the singular locus.

\item The only 0-curvature planes are those tangent to $T$.
\end{enumerate}
\end{mthm}

The singular orbifold locus of $SU(3)//SU(2)$ is a closed geodesic with a $\bbZ_3$ orbifold singularity. A neighborhood of a singular point has boundary homeomorphic to a suspension of a lens space $L(3;1)=S^3/\bbZ_3$. In particular, the underlying topological space is not a manifold.

Another interesting property of this orbifold is that it admits an isometric $S^1$ action, and we will see that

\begin{cor}\label{CorAlex}
The Alexandrov space $X^4=\calO^5/S^1=SU(3)//U(2)$ admits a metric of positive curvature.
\end{cor}

In this paper we also further study the orbifold quotients of Eschenburg spaces in a slightly more general way than done in \cite{FZOrbi}. We also provide some minor corrections and improvements, see Theorem \ref{ThmFix}. Furthermore, we show that a biquotient of the form $SU(3)//T^2$ admits a metric of positive curvature under a standard deformation iff there is some $S^1\subset T^2$ such that $SU(3)//S^1$ admits a metric of positive curvature under the same deformation.

This work is part of the author's Ph.D. thesis. The author would like to thank his advisor, Wolfgang Ziller, for his invaluable advice and encouragements. The author would also like to thank Jason DeVito, Matthew Tai and Martin Kerin for helpful conversations.

\section{Preliminaries}\label{SecPrelim}

Recall that an $n$-dimensional orbifold $\calO^n$ is a space modeled locally on $\bbR^n/\Gamma$ with $\Gamma\subset O(n)$ finite. Given a point $p\in\calO$, the orbifold group at $p$, which we'll denote as $\Gamma_p$ is the subgroup of the group $\Gamma$ in the local chart $\bbR^n/\Gamma$, that fixes a lift of $p$ to $\bbR^n$. Note that different choices of a lift of $p$ result in $\Gamma_p$ being conjugated, and as such, we will think of $\Gamma_p$ up to conjugacy.

Recall that a biquotient $X$ is of the form $X=G//U$ where $G$ is a Lie group and $U\subset G\times G$ acts as $(u_l,u_r)\cdot g = u_l\cdot g\cdot u_r^{-1}$ where $g\in G$ and $(u_l,u_r)\in U$. If $U=K\times H$ with $K,H\subset G$, then we can instead write $X=K\backslash G / H$. In his habilitation, Eschenburg showed that if $G//U$ is a manifold, then $\rk U \leq \rk G$. Since the argument is done on the Lie algebra level, we will see that it also holds when we allow $G//U$ to be an orbifold.

\begin{lem}\label{LemOrbiBi}
Let $\ft_\fu\subset \fu$ be a maximal abelian subalgebra. Then, $G//U$ is an orbifold if and only if for all non-zero $(X_1,X_2)\in\ft_\fu\subset\fu\subset\fg\oplus\fg$ and for all $g\in G$, $X_1 - \Ad(g)X_2 \neq 0$.

If $\pi:G\to G//U$ is the projection, and $g\in G$, then the orbifold group $\Gamma_{\pi(g)}\subset U$ is given by
\[
\Gamma_{\pi(g)} = \left\{
(h,k)\in U \mid hgk^{-1}g^{-1} = e
\right\}.
\]
\end{lem}

\begin{proof}
Let $M$ be a manifold, $\Gamma$ a Lie group, $\pi:M \to \calO=M/\Gamma$ the projection map. Then, for any $x\in M$, $\Gamma_{\pi(x)} = \Stab(x)$. As a corollary we get the description of $\Gamma_{\pi(g)}$ as desired. Observe that $G//U$ is an orbifold iff the stabilizer of every $g\in G$ is finite. The Lie algebra of $\Gamma_{\pi(g)}$ is $\{(X_1,X_2)\in\fu|X_1-\Ad(g)X_2=0\}$. Since we can conjugate $(X_1,X_2)\in\fu$ into an element of $\ft_{\fu}\subset\fu$, and since the stabilizer groups occur in conjugacy classes, the first claim follows as well.
\end{proof}

\begin{lem}\label{LemRk}
If $G//U$ is an orbifold, then $\rk \fu \leq \rk \fg$.
\end{lem}

\begin{proof}
Suppose $\rk \fu > \rk \fg$, let $\ft_\fu\subset \fu\subset \fg\oplus\fg$ be a maximal torus. Let $\phi_1,\phi_2$ be projections of $\fu$ onto the first and second copy of $\fg$ respectively. Pick a maximal torus $\ft_\fg\subset\fg$ such that $\phi_1(\ft_\fu)\subset \ft_\fg$.

Next, pick $g\in G$ such that $\Ad(g)\phi_2(\ft_\fu)\subset\ft_\fg$. This induces a linear map $\Phi:\ft_\fu\to\ft_\fg$ given by
\[
\Phi(X) = \phi_1(X)-\Ad(g)\phi_2(X).
\]
In particular, if $\rk\fu > \rk\fg$, we conclude that $\ker\Phi\neq\{0\}$. This implies that there exists $X\in\fu$ such that $\phi_1(X)-\Ad(g)\phi_2(X)=0$, which by the previous lemma implies that $G//U$ is not an orbifold.
\end{proof}

\begin{rmk}
For biquotients of $SU(n)$, we have a slightly more general form given by $SU(n)//H$ where $H\subset \{(A,B)\in U(n)\times U(n)\mid \det(A)=\det(B)\}$. Such a group $H$ clearly still acts on $SU(n)$, and we again denote the quotient by $SU(n)//H$.
\end{rmk}

An important tool in section \ref{SecO5} is Cheeger deformations \cite{CDeform}, which is used to improve the curvature on manifolds and orbifolds. To perform a Cheeger deformation along a subgroup $K\subset G$, choose $\lambda>0$, and define $(G,g_\lambda)=G\times_K \lambda K$, where $G$ is equipped with a bi-invariant metric and $\lambda K$ is equipped with the induced metric scaled by $\lambda$. The metric $g_\lambda$ is still left-invariant, but is no longer bi-invariant; however, it is right $K$-invariant. Let $\fk$ be the Lie algebra of $K$, and let $X^\fk$ denote the $\fk$ component of $X\in\fg$. We will use the following result of Eschenburg \cite{EHab}

\begin{lem}[Eschenburg]
If $(G,K)$ is a compact symmetric pair, equip $G$ with the metric induced by the map $G\times\lambda K\to G$, given by $(g,k)\mapsto gk^{-1}$. The metric has non-negative sectional curvature, and $\sec(X,Y)=0$ iff $[X,Y]=[X^\fk,Y^\fk]=0$.
\end{lem}

We also remark that for the sake of brevity, we often omit zero entries in matrices when no confusion about the dimension of the matrix can arise.

\section{Proof of Theorem \ref{ThmClass}}\label{SecPfClass}

The first class (the homogeneous spaces) are well known. One simply classifies connected subgroups of $SU(3)$, which up to conjugation are $U(2), T^2, SU(2), SO(3)$ and $S^1_{p,q}$, where $S^1_{p,q}=\diag(z^p,z^q,\bar{z}^{p+q})$ with $p,q\in\bbZ$. We may assume without loss of generality that $p\geq q\geq 0$. Throughout the rest of this section, we assume that $G//U$ is not given by a homogeneous action, in particular, $U$ must act on both sides.

Recall that the subgroups of $SU(3)$ other than $S^1$ are unique up to conjugation. For $SU(2)$ and $U(2)$ we will use the standard upper-left block embeddings $\diag(A,1)$ for $SU(2)$ and $\diag(A,\bar{\det(A)})$ for $U(2)$. For $SO(3)$ we utilize a convenient, although non-standard embedding of $SO(3)$ into $SU(3)$. On the Lie algebra level, we have
\begin{equation}\label{SubAlg}
\so(3) = \left\{\left.
\begin{pmatrix}
ai & 0 & z\\
0 & -ai & -\bar{z}\\
-\bar{z} & z & 0
\end{pmatrix}
\right|
a\in\bbR,z\in\bbC
\right\}
\end{equation}
this embedding of $\so(3)$ is given by conjugating the standard embedding of $SO(3)\subset SU(3)$ and $\so(3)\subset\su(3)$ by
\begin{equation}\label{Conj}
g_0=\begin{pmatrix}
\frac{1}{\sqrt{2}} & \frac{-i}{\sqrt{2}} & 0\\
\frac{1}{\sqrt{2}} & \frac{i}{\sqrt{2}} & 0\\
0 & 0 & -i
\end{pmatrix}.
\end{equation}
The advantage of this embedding is that it has a convenient maximal torus, simplifying some computations. We also note that throughout this paper in the examples which involve $SO(3)$, should the reader desire to utilize the standard $SO(3)\subset SU(3)$, the results stated for $X\in SU(3)$, should now be interpreted as being about $X\cdot g_0$.

From Lemma \ref{LemRk}, we know that if $SU(3)//U$ is an orbifold, then $\rk\fu\leq 2$. In particular, we must have $U=S^1,T^2,SU(2),SO(3),U(2),SU(2)\times S^1$, $SO(3)\times S^1$, a finite quotient of $SU(3), Sp(2)$ or $SU(2)\times SU(2)$, or the exceptional group $G_2$.

The cases $\fu = \su(3),\sp(2),\fg_2$ or $\su(2)\oplus\su(2)$ can be ruled out quickly. Observe $\dim SU(3)//(SU(3)/\Gamma) = 0$, and in particular the action must be homogeneous. Additionally, $\dim SU(3) < \dim (Sp(2)/\Gamma)$, and $\dim SU(3) < \dim G_2$, so there can be no orbifold biquotients of the form $SU(3)//(Sp(2)/\Gamma)$ or $SU(3)//G_2$. Recall that the only embeddings of $A_1\oplus A_1\hookrightarrow \su(3)\times\su(3)$ map an $A_1$ factor into each $\su(3)$ factor. However, the maximal torii of the possible $A_1$ embedding are conjugate, which violates the conditions of Lemma \ref{LemOrbiBi}.

The cases when $U$ is either $S^1$ or $T^2$ yield the  7 and 6 dimensional Eschenburg spaces respectively, and both are covered in section \ref{SecEsch}.

Next, suppose $U=SO(3)$. Since $U$ acts on both sides and there is a unique up to conjugation embedding of $SO(3)$ into $SU(3)$, we must have $U=\Delta SO(3)\subset SU(3)\times SU(3)$. However, this obviously leads to a violation of Lemma \ref{LemOrbiBi}. Therefore, there are no non-homogeneous orbifolds of the form $SU(3)//SO(3)$.

Next, suppose $U=SO(3)\times S^1$. Recall that $SO(3)\times S^1$ is not a subgroup of $SU(3)$. Therefore the $S^1$ and the $SO(3)$ must act on different sides. We get the family of orbifolds $S^1_{p,q}\backslash SU(3)/SO(3)$ whose precise orbifold structure we discuss in section \ref{SecNew}.

Next, we let $U=SU(2)$. By the same argument as for $SO(3)$, we cannot have $U=\Delta SU(2)\subset SU(3)\times SU(3)$. The only remaining non-trivial embedding is if we map $U$ to $SU(2)$ on one side and to $SO(3)$ on the other. We study this embedding on the Lie algebra level. In particular, we use the following bases:
\begin{equation}\label{Bases}
\begin{aligned}
&\su(2):&
I_1 &= \begin{pmatrix}
       i & 0 & 0\\
       0 & -i & 0\\
       0 & 0 & 0
       \end{pmatrix}&
J_1 &= \begin{pmatrix}
       0 & 1 & 0\\
       -1 & 0 & 0\\
       0 & 0 & 0
       \end{pmatrix}&
K_1 &= \begin{pmatrix}
       0 & i & 0\\
       i & 0 & 0\\
       0 & 0 & 0
       \end{pmatrix}\\
&\so(3):\qquad&
I_2 &= \begin{pmatrix}
       2i & 0 & 0\\
       0 & -2i & 0\\
       0 & 0 & 0
       \end{pmatrix}&
J_2 &= \begin{pmatrix}
       0 & 0 & \sqrt{2}\\
       0 & 0 & -\sqrt{2}\\
       -\sqrt{2} & \sqrt{2} & 0
       \end{pmatrix}&
K_2 &= \begin{pmatrix}
       0 & 0 & i\sqrt{2}\\
       0 & 0 & i\sqrt{2}\\
       i\sqrt{2} & i\sqrt{2} & 0
       \end{pmatrix}
\end{aligned}
\end{equation}
Under these bases we have: $[I_n,J_n]=2K_n,[J_n,K_n]=2I_n,[K_n,I_n]=2J_n$. We define the embedding $\phi:\su(2)\hookrightarrow\su(3)\oplus\su(3)$ as $\phi(I)=(I_1,I_2)$, $\phi(J)=(J_1,J_2)$ and $\phi(K)=(K_1,K_2)$. Where $I,J,K$ is the standard quaternionic basis for $\su(2)=\sp(1)=\Im\,\bbH$. To verify that the resulting biquotient is an orbifold choose $\{tI|t\in\bbR\}$ as a maximal torus of $\su(2)$, then $\phi_1(tI)=tI_1,\phi_2(tI)=tI_2$. The condition we need to verify is that $tI_1-\Ad(g)tI_2=0$ iff $t=0$, but $I_2=2I_1$, so we have $tI_1 = 2t\Ad(g)I_1$, so $I_1=2\Ad(g)I_1$ if $t\neq 0$, but conjugation preserves the norm, so we must have $t=0$. Therefore the resulting biquotient is an orbifold of dimension 5, which we denote by $\calO^5$. For the sake of convenience, we use $SU(2)_\phi\subset SU(3)\times SU(3)$ to denote this embedding of $SU(2)$, and we write $SU(3)//SU(2)_{\phi}$ for the resulting biquotient. We study the orbifold structure of $SU(3)//SU(2)_\phi$ in section \ref{SecNew} and its metric properties in section \ref{SecO5}.

The last two cases, $U=SU(2)\times S^1$ and $U=U(2)=(SU(2)\times S^1)/\bbZ_2$, we consider jointly. The first observation is that in both cases $SU(2)\subset U$. In particular, $SU(3)//SU(2)$ must be an orbifold, where $SU(2)\subset U$. Therefore, $SU(2)$ either acts on only one side, or on both by the above $\phi:SU(2)\to SU(3)\times SU(3)$. Suppose it is the latter, then there is only one choice of $S^1$ which commutes with $SU(2)_\phi$, namely $\diag(z,z,\bar{z}^2)$ acting on the left. We will now show that this does not result in an orbifold.

Consider $\diag(i,i,-2i)$ in the tangent space of the $S^1$ component, and using the bases in (\ref{Bases}), $(I_1,I_2)$ in the tangent space of the $SU(2)$ component. For the sum, we have $\diag(2i,0,-2i)$ on the left, and $\diag(2i,-2i,0)$ on the right. These two elements of $\su(3)$ are clearly conjugates, therefore, $SU(3)//(SU(2)_\phi\times S^1)$ is not an orbifold by Lemma \ref{LemOrbiBi}.

Finally, we consider the case where $U=SU(2)\times S^1$ or $U=U(2)$ and $SU(2)$ acts only on one side (we choose the right for convenience). We claim that in this case $SU(3)//U$ is a weighted projective space. Recall that a weighted projective space is defined as $\CP^2[\lambda_0,\lambda_1,\lambda_2] = S^5/S^1$ where the $S^1$-action is given by
\begin{equation}\label{WCP}
w\star(z_0,z_1,z_2) = (w^{\lambda_0} z_0, w^{\lambda_1} z_1, w^{\lambda_2} z_2),
\end{equation}
where $\lambda_i\in\bbZ\setminus\{0\}$, $\gcd(\lambda_0,\lambda_1,\lambda_2)=1$ and $(z_0,z_1,z_2)$ are coordinates on $\bbC^3\supset S^5$. For $SU(3)/SU(2)$, the possible $S^1$ biquotient actions are parametrized by $p,q,r\in\bbZ$, $\gcd(p,q,r)=1$, none of $p+q,p+r,q+r$ zero, and are induced by the following action on $SU(3)$:
\[
z\star X = \diag(z^p,z^q,z^r) X \diag(1,1,z^{p+q+r})^{-1},\qquad X\in SU(3).
\]
Recall that under our chosen representation of $SU(2)\subset SU(3)$, we have a well-behaved projection map $\pi:SU(3)\to SU(3)/SU(2)=S^5\subset \bbC^3$. Namely,
\[
\pi\begin{pmatrix}
x_{11} & x_{12} & x_{13}\\
x_{21} & x_{22} & x_{23}\\
x_{31} & x_{32} & x_{33}
\end{pmatrix} = (x_{13}, x_{23}, x_{33}) \in \bbC^3.
\]
Under this identification, the above action becomes
\[
z\star(x_{13},x_{23},x_{33}) = (\bar{z}^{q+r}x_{13},\bar{z}^{p+r}x_{23},\bar{z}^{p+q}x_{33})
\]
and hence the quotient is the weighted projective space $\CP^2[-q-r,-p-r,-p-q]\cong\CP^2[q+r,p+r,p+q]$. A note of caution is that this representation need not be in lowest terms, and for proper representation as a weighted $\CP^2$, we need to divide all three weights by their greatest common divisor and normalize the signs to be positive. In general, proper choices of $p,q,r$ allow us to obtain any weighted projective space. As a second note, it does not matter whether the action of $S^1\times SU(2)$ is effective, so this also covers the case of $SU(3)//U(2)$, which corresponds to the case when $p,q,r$ are all odd.

\section{Orbifold Structure of the New Examples}\label{SecNew}

The singular locus of the generalized Eschenburg spaces will be studied in section \ref{SecEsch}. We will now study the singular locus for the remaining two cases, and start with $\calO^5=SU(3)//SU(2)_\phi$.

\begin{prop}
$\calO^5=SU(3)//SU(2)_\phi$ as defined above has a closed geodesic as its singular locus, and each point on the singular locus has an order 3 orbifold group.
\end{prop}

\begin{proof}
Let $\pi:SU(3)\to \calO$ denote the projection map. Let $\phi_1,\phi_2$ be the projections from $\fu$ onto $\su(2)$ and $\so(3)$ respectively, and $\psi_1,\psi_2$ projections from $U=SU(2)_\phi\subset SU(3)\times SU(3)$ onto $SU(2)$ and $SO(3)$ respectively. Let $g\in SU(3)$ and $h\in U$ be an element other than identity.

Suppose that $\psi_1(h)\cdot g\cdot \psi_2(h)^{-1} = g$ (i.e. $g$ has non-trivial stabilizer). Since stabilizer groups occur in conjugacy classes, we can assume that $h$ lies in the maximal torus $h=e^{tI} = (e^{t I_1},e^{tI_2})$ of $SU(2)_\phi$. Since $\psi_i(e^X)=e^{\phi_i(X)}$, the condition reduces to
\[
g^{-1} \begin{pmatrix}
e^{ti} &&\\
&e^{-ti}&\\
&&1
\end{pmatrix} g = \begin{pmatrix}
e^{2ti} &&\\
&e^{-2ti}&\\
&&1
\end{pmatrix},
\]
Since a conjugation can only permute eigenvalues, we see that either $e^{ti}=e^{2ti}$ or $e^{ti}=e^{-2ti}$. The first case is degenerate, since it implies that $e^{ti}=1$, and so $h$ is identity. In the second case we get that $e^{ti}$ is a third root of unity and
\[
g=\begin{pmatrix}
&u&\\
v&&\\
&&w
\end{pmatrix},\qquad uvw=-1
\]

We now show that $\pi(g)$ lies in a single circle for $g$ as above. Re-write
\[
g = \begin{pmatrix}
&e^{\lambda i} &\\
-e^{-(\lambda+\mu)i}&&\\
&&e^{\mu i}
\end{pmatrix}
\]
and act on $g$ by $e^{\lambda/3 I}\in U$.
\begin{align*}
g\to
\begin{pmatrix}
e^{-\lambda i/3} &&\\
&e^{\lambda i/3}&\\
&&1
\end{pmatrix}&
\begin{pmatrix}
&e^{\lambda i}&\\
-e^{-(\lambda+\mu)i}&&\\
&&e^{\mu i}
\end{pmatrix}
\begin{pmatrix}
e^{2\lambda i/3} &&\\
&e^{-2\lambda i/3}&\\
&&1
\end{pmatrix}\\
&= \begin{pmatrix}
&1&\\
-e^{-\mu i}&&\\
&&e^{\mu i}
\end{pmatrix}
\end{align*}
Define
\[
g_z =
\begin{pmatrix}
&1&\\
-\bar{z}&&\\
&&z
\end{pmatrix}.
\]

So far, we have shown that each singular orbit contains an element of the form $g_z$. Computations show that $\pi(g_z)=\pi(g_w)$ iff $w=\pm z$. Therefore, we conclude that the image of the singular orbits under $\pi$ forms a circle. The above also shows that each element in the stabilizer of $g_z$ has order 3, and hence $|\Gamma_{\pi(g_z)}|=3^n$.

Recall that every group of order $p^n$ where $p$ is prime has a non-trivial center. Additionally, all elements of order 3 in $SU(2)$ are conjugate to each other. Therefore, the stabilizer is abelian. Furthermore, given an element of order 3 inside $SU(2)$, it commutes only with the elements in the same maximal torus. Therefore, the orbifold group is precisely $\bbZ_3\subset SU(2)$.
\end{proof}

\begin{rmk}
We recall that there is a unique smooth 3-dimensional lens space of the form $S^3/\bbZ^3 = L(3;1) = L(3;2)$. Therefore, it is the space of directions normal to the singular locus.
\end{rmk}

Next we examine the singular locus of the quotients of the Wu manifold.

\begin{prop}\label{PropStructWu}
The quotient $\calO_{p,q}=S^1_{p,q}\backslash SU(3)/SO(3)$ is an orbifold iff $p\geq q > 0$. Furthermore, the action is effective iff $(p,q)=1$. Its singular locus consists of a singular $\RP^2$ with orbifold group $\bbZ_2$, with possibly one point on it with a larger orbifold group, and up to two other isolated singular points. The orbifold group at the singular points are $\bbZ_p,\bbZ_q$, and $\bbZ_{p+q}$.
\end{prop}

\begin{proof}
To find the singular locus we need to see when $\diag(z^p,z^q,\bar{z}^{p+q})$ is conjugate to an element of $SO(3)$. Without loss of generality, we only need to check when it is conjugate to something in the maximal torus $T\subset SO(3)$. With our choice of $SO(3)$, the most convenient maximal torus has the form $\diag(w,\bar{w},1)$

The conjugacy class is determined by the eigenvalues, and hence $\diag(z^p,z^q,\bar{z}^{p+q})\in S^1_{p,q}$ is conjugate to $\diag(w,\bar{w},1)\in SO(3)$ iff $z^p,z^q$, or $\bar{z}^{p+q}$ is equal to 1. In particular, each of the three choices yields an orbifold group of order $p,q$, and $p+q$ respectively.

Let $g\in SU(3)$ be a preimage of an orbifold point w.r.t. the action of $S^1_{p,q}$, i.e. $g\cdot \diag(z^p,z^q,\bar{z}^{p+q})\cdot g^{-1}=\diag(w,\bar{w},1)$. We first consider the case where $w^2\neq 1$. In this case all three eigenvalues are distinct, and so the two diagonal matrices are related by a permutation matrix, i.e. $g=g_i$ as defined below.

\begin{align*}
g_1&=\begin{pmatrix}
1 &&\\
&1&\\
&&1
\end{pmatrix}&
g_4&=\begin{pmatrix}
&-1&\\
-1&&\\
&&-1
\end{pmatrix}\\
g_2&=\begin{pmatrix}
&&1\\
1&&\\
&1&
\end{pmatrix}&
g_5&=\begin{pmatrix}
&&-1\\
&-1&\\
-1&&
\end{pmatrix}\\
g_3&=\begin{pmatrix}
&1&\\
&&1\\
1&&
\end{pmatrix}&
g_6&=\begin{pmatrix}
-1&&\\
&&-1\\
&-1&
\end{pmatrix}
\end{align*}
We note that given $g_i$, every element of the form $\diag(\rho,\eta,\zeta)X_i$ lies in the same orbit as $g_i$. For example, for $g_1$, let $\bar{z}$ be a $(p+q)^{th}$ root of $\zeta$, and $w=\rho\bar{z}^p$. Then, $\diag(z^p,z^q,\bar{z}^{p+q})\in S^1$, $\diag(w,\bar{w},1)\in SO(3)$, and $\diag(z^p,z^q,\bar{z}^{p+q})\diag(w,\bar{w},1)=\diag(\rho,\eta,\zeta)$.

Recall that we are using a non-standard $SO(3)\subset SU(3)$, and as such $g_1,g_4\in SO(3)$, since, $g_4 = g_0\cdot\diag(-1,1,-1)\cdot g_0^{-1}$, where $g_0$ is as in (\ref{Conj}). This implies that $g_5 = g_2 g_4\in g_2 SO(3),g_6 = g_3 g_4\in g_3 SO(3)$. Corresponding to three (possibly) isolated singular points.

Now, suppose that $w^2=1$ and $w\neq 1$, i.e. $w=-1$. This implies that 2 of $z^p,z^q,\bar{z}^{p+q}$ are -1, and the third is 1. Without loss of generality, assume that $z^p=z^q=-1,\bar{z}^{p+q}=1$. This implies that either $(p,q)>1$ or $z=-1$. The former contradicts the assumption that $S^1$ acts effectively, and hence $z=-1$. Next note that exactly one of $p,q,p+q$ is even. in what follows, we assume that $p+q$ is the even exponent.

We now have $\diag(-1,-1,1)\cdot h\cdot \diag(-1,-1,1) = h$, so $h$ commutes with $\diag(-1,-1,1)$ and so
\[
h \in U(2) = \left\{\left.\begin{pmatrix}
A &\\
& \bar{\det(A)}
\end{pmatrix}\right| A\in U(2)\right\}\subset SU(3).
\]
To determine $\pi(U(2))$, we need to find the subgroup $K\subset S^1_{p,q}\times SO(3)$ which preserves $U(2)$. Since $S^1\subset U(2)$, we must have $K=S^1_{p,q}\times(SO(3)\cap U(2))$. The intersection is
\[
SO(3)\cap U(2)=
\left\{\begin{pmatrix}
w & &\\
& \bar{w} &\\
& & 1
\end{pmatrix}\right\}\cup
\left\{\begin{pmatrix}
& z &\\
\bar{z} &  &\\
& & -1
\end{pmatrix}\right\}.
\]
Indeed, for $g_0$ as in (\ref{Conj}), we have $g_0\in U(2)$, and hence $SO(3)\cap U(2) = g_0(SO(3)_{std}\cap U(2))g_0^{-1} = g_0 O(2)g_0^{-1}$, where $O(2) = \{\diag(A,\det(A))|A\in O(2)\}\subset SU(3)$.

Hence, $K$ is a disjoint union of two copies of $T^2$. Identifying $U(2)\subset SU(3)$ with the upper $2\times 2$ block, we can rewrite the action of $K$ on $U(2) = \left\{\left.A = \begin{pmatrix}
a & b\\
c & d
\end{pmatrix}\right| A\in U(2)\right\}$. Restricting to the identity component of $K$, we get
\[
(z,w) \star A = \begin{pmatrix}
z^p &\\
& z^q
\end{pmatrix}
\begin{pmatrix}
a & b\\
c & d
\end{pmatrix}
\begin{pmatrix}
w &\\
& \bar{w}
\end{pmatrix}^{-1}.
\]
Using an appropriate element $z\in S^1_{p,q}$, we can assume that $\det(A)=1$, and hence $z^{p+q}=1$. Thus the quotient of this action is the same as $SU(2)//(\bbZ_{p+q}\times S^1)$, given by
\[
(z,w) \star \begin{pmatrix}
a & -\bar{b}\\
b & \bar{a}
\end{pmatrix} =
\begin{pmatrix}
z^p &\\
& z^q
\end{pmatrix}
\begin{pmatrix}
a & -\bar{b}\\
b & \bar{a}
\end{pmatrix}
\begin{pmatrix}
w & \\
& \bar{w}
\end{pmatrix}^{-1}
\]
where $z^{p+q}=1$ and $|a|^2+|b|^2=1$. Identifying $SU(2)$ with $S^3\subset \bbC^2$ via $\begin{pmatrix}
a & -\bar{b}\\
b & -a
\end{pmatrix}\mapsto(a,b)$ the $S^1$ action by $w$ becomes $w\cdot (a,b)\to(aw,bw)$. This is the Hopf action, and hence $S^3/S^1=S^2$ with projection $S^3\subset\bbC^2\to S^2=\bbC\cup\{\infty\}$ given by $(a,b)\mapsto ab^{-1}$. The action by $z$ then becomes $z\cdot(a,b)=(z^pa,z^qb)=(z^pa,\bar{z}^pb)$, since $z^{p+q}=1$. This induces an action on $S^2$ given by $ab^{-1}\mapsto z^{2p}(ab^{-1})$. Notice that $z=-1$ acts trivially corresponding to the fact that $\bbZ_2$ fixed $U(2)$. Thus we have rotation by $2\pi p / (p+q)$, since $z$ runs over the $(p+q)^{th}$ roots of unity.

Finally, we must consider the second component of $K$, which can be considered as the action on $U(2)$ by $\begin{pmatrix}
& 1\\
-1 &
\end{pmatrix}$ on the right. On $S^3$, this action corresponds to $(a,b)\mapsto (\bar{b},-\bar{a})$, and on $S^2 = \bbC\cup\{\infty\}$ we get $x\mapsto -1/\bar{x}$, which is precisely the antipodal map. Thus, $\pi(U(2))$ is a (possibly singular) $\RP^2$, and the image of 0 and $\infty$ is the only orbifold point with orbifold group $\bbZ_{|p+q|}$.
\end{proof}

\begin{rmk}
We note that the singular $\RP^2$ above has a distinguished point, which has a larger orbifold group unless the even integer among $p,q,p+q$ is equal to $\pm 2$.
\end{rmk}

\section{$SU(3)//SU(2)_\phi$}\label{SecO5}

In this section we study the curvature of the orbifold $SU(3)//SU(2)_\phi$, and prove Theorem \ref{ThmO5}.

The most natural metric on $\calO = SU(3)//SU(2)_\phi$ is induced by the bi-invariant metric on $SU(3)$. Using this metric, we get

\begin{prop}
The orbifold $\calO^5=SU(3)//SU(2)_\phi$, equipped with the metric induced by the bi-invariant metric on $SU(3)$, has quasi-positive curvature.
\end{prop}

In fact, the image of identity has positive curvature. The downside of this metric is that every singular point has planes of zero-curvature. The proof of this proposition is straightforward, but tedious. We will omit it since we will now show how to improve this metric using a Cheeger deformation.

Recall that in the construction of $\calO^5$ we use a non-standard $SO(3)$, see (\ref{SubAlg}), and we utilize the bases as of $\so(3)$ and $\su(2)$ as in (\ref{Bases}). We now let
\[
\fk = \so(3)\qquad and\qquad \fh=\su(2).
\]

We apply a Cheeger deformation along $SO(3)\subset SU(3)$, which results in a left-invariant, right $SO(3)$-invariant metric. As such, the $SU(2)_\phi$ acts by isometries, so the deformation induces a new metric on $\calO^5=SU(3)//SU(2)_\phi$.

\begin{thm}\label{ThmChDef}
$\calO$ with the metric induced by a Cheeger deformation along $SO(3)\subset SU(3)$ has the following properties:

\begin{enumerate}
\item $\calO$ has almost positive curvature.

\item The set of points with 0-curvature planes forms a totally geodesic flat 2-torus $T$ that is disjoint from the singular locus.

\item Each point in $T$ has exactly one 0-curvature plane, and those planes are tangent to $T$.
\end{enumerate}

\end{thm}

\begin{proof}
As in section \ref{SecPrelim}, we denote the bi-invariant metric on $SU(3)$ by $\langle\cdot,\cdot\rangle$ and the Cheeger deformed metric by $\langle\cdot,\cdot\rangle_\lambda$. Left translations are isometric in $\langle\cdot,\cdot\rangle_\lambda$, and hence we can identify tangent vectors at $g\in SU(3)$ with vectors in $\su(3)$. Under this identification, the vertical space become $\{\psi(C) - \Ad(g^{-1})C\mid C\in\fh\}$, where $\psi:\fh\to\fk$ is defined by $I_1\mapsto I_2,J_1\mapsto J_2$ and $K_1\mapsto K_2$. Then a vector $X\in\su(3)$ is horizontal at $g$ iff $\langle X, \psi(C) - \Ad(g^{-1})C\rangle_\lambda=0$ for all $C\in\fh$. Also recall that in $\langle\cdot,\cdot\rangle_\lambda$, a plane spanned by $A,B\in \su(3)$ is flat iff $[A,B]=[A^\fk,B^\fk]=0$. Additionally, for some computations in this proof, we recall that $\langle X,Y\rangle_\lambda = \langle X^{\fk^\perp}, Y^{\fk^{\perp}}\rangle + \frac{\lambda}{1+\lambda}\langle X^\fk , Y^\fk \rangle$. We will use $\nu = \frac{\lambda}{1+\lambda}\in(0,1)$ for brevity.

Suppose that at the image of some point $g\in SU(3)$, we have 0-curvature. Let $A,B$ be two elements of $T_g^h$ (left translated to the identity), which span a 0-curvature plane.

We begin by making a series of claims:

\begin{enumerate}
\item We may assume $A^\fk=0$.

Indeed, if $A^\fk\neq0$ and $B^\fk\neq0$, then $[A^\fk,B^\fk]=0$ iff $A^\fk=c\cdot B^\fk$ since $\fk$ has rank one.

\item $\langle \Ad(g)A,\fh\rangle = 0$.

Observe that since $\langle A,\fk\rangle = 0$, we have $\langle \Ad(g) A, \fh\rangle = \langle A, \Ad(g^{-1})\fh\rangle = \langle A, \Ad(g^{-1})\fh\rangle_\lambda$. If $\langle A, \Ad(g^{-1})(C)\rangle_\lambda\neq0$, then, since $A\in\fk^\perp$, we have $\langle A, \psi(C)- \Ad(g^{-1})C\rangle_\lambda\neq0$, so $A$ is not horizontal.

\item $B^\fk\neq 0$.

Suppose that $B^\fk = 0$, then by the same argument as we used with $A$, we must have $\langle \Ad(g) B, \fh\rangle = 0$. In particular, both $\Ad(g) A$ and $\Ad(g) B$ are horizontal with respect to the submersion $SU(3) \to SU(3)/SU(2) = S^5$. Since, endowed with the metric induced by $\langle\cdot,\cdot\rangle$, $S^5$ has positive sectional curvature, it follows that $0\neq [\Ad(g)A,\Ad(g)B] = \Ad(g)[A,B]$, in particular, $[A,B]\neq 0$.

\item Up to scaling $A=\Ad(u)X$ for some $u\in SO(3)$, where $X=\diag(i,i,-2i)$.

Observe that $B^\fk=t\cdot\Ad(u)I_2$ for some non-zero $t\in\bbR$ and $u\in SO(3)$. Notice that since $A\in\fk^\perp$, it follows that $[A,B^\fk]\in\fk^\perp$ and since $(\fg,\fk)$ is a symmetric pair, $[A,B^{\fk^\perp}]\in\fk$, and hence we must have $[A,B^\fk]=[A,B]^{\fk^\perp}=0$. Therefore, $0 = [A,B^\fk] = t\cdot\Ad(u)[\Ad(u^{-1})A,I_2]$. A matrix that commutes with $I_2$ must be diagonal. Since $A\in\fk^\perp$, and $u\in SO(3)$, $\Ad(u^{-1})A\in\fk^\perp$ as well. In particular, $\Ad(u^{-1})A$ is also orthogonal to $I_2$, which implies our claim. Furthermore, since scaling $A$ does not change the plane spanned by $A,B$ we will assume that $A = \Ad(u)X$.

\item By changing the point in the orbit, we may assume that $g\in U(2)$.

First, we observe that $\langle \Ad(gu)X,\fh\rangle = \langle \Ad(g)A,\fh\rangle = 0$, since $A=\Ad(u)X$.

Let
\[
gu = \begin{pmatrix}
a_1 & a_2 & a_3\\
b_1 & b_2 & b_3\\
c_1 & c_2 & c_3
\end{pmatrix}
\qquad\text{and}\qquad
\Ad(gu)X = \begin{pmatrix}
m_{11} & m_{12} & m_{13}\\
m_{21} & m_{22} & m_{23}\\
m_{31} & m_{32} & m_{33}
\end{pmatrix}.
\]
Since this has to be orthogonal to $\fh$ with respect to the bi-invariant metric, we conclude that $m_{13}=0,m_{23}=0,m_{11}=m_{22}$ and $m_{33}=-2 m_{11}$. We compute that $m_{13} = (a_1 \bar{c_1} + a_2\bar{c_2} -2 a_3\bar{c_3})i= (-3a_3\bar{c_3})i$. This is zero iff $a_3=0$ or $b_3=0$. Similarly, $m_{23}=0$ implies $b_3=0$ or $c_3=0$. We also compute that $m_{11} = (|a_1|^2 + |a_2|^2 -2|a_3|^2)i = (1-3|a_3|^2)i$ and $m_{22} = (|b_1|^2 + |b_2|^2 - 2|b_3|^2)i = (1-3|b_3|^2)i$. Since $m_{11}=m_{22}$, we must have $|b_3|=|a_3|$. Finally, we see that $m_{33}+2m_{11} = (1-|c_3|^2) + 2(1-3|a_3|^2) = 0$. Suppose that $c_3=0$, then $|a_3|^2+|b_3|^2=1$, so $|a_3|^2=1/2$, but this implies that $1+(2-3\cdot 1/2) = 3/2\neq 0$, so $m_{33}\neq -2 m_{11}$. Therefore, $c_3\neq 0$, which implies that both $a_3$ and $b_3$ are zero, which also implies $c_1=c_2=0$. Therefore, $gu\in U(2)\subset SU(3)$, and so $g\in SU(2)\diag(w,w,\bar{w}^2)SO(3)$. Thus, $g = u\cdot \diag(w,w,\bar{w}^2)\cdot v$ with $u\in SU(2),v\in SO(3)$. Since there exists $x\in SU(2)$ such that $(x,v)\in SU(2)_\phi$, we can change the point in the orbit and assume that $g\in U(2)$.

\item We may assume that $A = X = \diag(i,i,-2i)$.

Since $\langle A, \Ad(g^{-1})\fh\rangle = 0$ and $g\in U(2)$, it follows that $A\in \fh^\perp$, and as we saw, $A\in\fk^\perp$ as well. Thus,
\[
A\in \fk^\perp\cap \fh^\perp = \left\{\left.
\begin{pmatrix}
ti & 0 & z\\
0 & ti & \bar{z}\\
-\bar{z} & -z & -2ti
\end{pmatrix}
\right| t\in\bbR,z\in\bbC
\right\}.
\]
We furthermore know that $A$ is conjugate to $\diag(i,i,-2i)$ and so has eigenvalues $i,i,-2i$. In particular, there is a repeated pair. The eigenvalues of $A$ as above are:
\[
ti,-ti/2\pm1/2\sqrt{-9t^2-8|z|^2}
\]
The last two are equal iff $t=z=0$, which means we have $A=0$. Therefore, the repeated eigenvalue is $ti$. So, the determinant must be $(ti)^2(-2ti) = 2t^3i$.
Computing the determinant of $A$, we have $\det(A) = 2t^3i + 2t|z|^2i$. Therefore $z=0$, and hence $A=X$.
\end{enumerate}

We now examine the possible values of $B$. Since $[A,B]=0$, this implies that $B = \begin{pmatrix}ri&z\\
-\bar{z}&-(r+s)i&\\
&&si\end{pmatrix}$, where $r,s\in\bbR$, $z\in\bbC$. We can also assume that $s=0$ by replacing $B$ with $B+(s/2)A$.

The question now becomes when a $B$ of this form is a horizontal vector at
\[
g=\begin{pmatrix}a&b&\\
-\bar{b}&\bar{a}&\\
&&1
\end{pmatrix}\begin{pmatrix}w&&\\
&w&\\
&&\bar{w}^2
\end{pmatrix}.
\]

From $\langle B, I_2-\Ad(g^{-1})I_1\rangle_{\lambda}=0$, we get:

\begin{align*}
0 &=\langle B, I_2 - \Ad(g^{-1})I_1\rangle_{\lambda}\\
&=
\left\langle
\begin{pmatrix}
ri & z &\\
-\bar{z} & -ri &\\
&& 0
\end{pmatrix},
\begin{pmatrix}
(2+|b|^2 - |a|^2)i & -2i\bar{a}b & \\
-2i \bar{b}a & (-2 + |a|^2 - |b|^2)i &\\
&&0
\end{pmatrix}
\right\rangle_\lambda\\
&= \nu r (2+|b|^2-|a|^2) + (-\bar{az}b i + az\bar{b} i)\\
&= \nu r (3|b|^2 + |a|^2) + 2Im(\bar{az}b).
\end{align*}
So, $r = \dfrac{-2 Im(\bar{az}b)}{\nu(|a|^2 + 3|b|^2)}$.

If we plug in $a=0$ or $b=0$, we get $r=0$, so $B^{\fk}=0$, which contradicts one of our earlier observations. So we may assume that $a\neq0,b\neq0$. Under these assumptions, plugging in what we obtained for $r$, we get:

\begin{align*}
0 &= \langle B, J_2 - \Ad(g^{-1})J_1\rangle_{\lambda}\\
&= \left\langle
\begin{pmatrix}
\dfrac{-2 Im(\bar{az}b)}{\nu(|a|^2 + 3|b|^2)} i & z &\\
-\bar{z} & \dfrac{2 Im(\bar{az}b)}{\nu(|a|^2 + 3|b|^2)} i &\\
&&0
\end{pmatrix},
\begin{pmatrix}
-ab+\bar{ab} & -b^2 - \bar{a}^2 & \sqrt{2}\\
a^2 + \bar{b}^2 & ab-\bar{ab} & -\sqrt{2}\\
-\sqrt{2} & \sqrt{2} & 0
\end{pmatrix}
\right\rangle_\lambda\\
&=\frac{-1}{2(|a|^2+3|b|^2)} \left(3 |a|^2 b^2 \bar{z} + a^2 |b|^2 z + \bar{a}^2 |b|^2 \bar{z} + 3 |a|^2 \bar{b}^2 z + 3 |b|^2 b^2 \bar{z} + |a|^2 \bar{a}^2 \bar{z} + |a|^2 a^2 z + 3 |b|^2 \bar{b}^2 z\right)\\
&= \frac{-1}{2(|a|^2+3|b|^2)}(|a|^2+|b|^2)(a^2 z + 3 b^2 \bar{z} + \bar{a}^2 \bar{z} + 3 \bar{b}^2 z) = \frac{-1}{|a|^2+3|b|^2} Re(a^2 z + 3 b^2 \bar{z}).
\end{align*}
So, $Re(a^2 z + 3 b^2 \bar{z})=0$.

Similarly,

\begin{align*}
0&=\langle B, K_2 - \Ad(g^{-1})K_1\rangle_{\lambda}\\
&= \left\langle
\begin{pmatrix}
\dfrac{-2 Im(\bar{az}b)}{\nu(|a|^2 + 3|b|^2)} i & z &\\
-\bar{z} & \dfrac{2 Im(\bar{az}b)}{\nu(|a|^2 + 3|b|^2)} i &\\
&&0
\end{pmatrix},
\begin{pmatrix}
(ab+\bar{ab})i & (b^2 - \bar{a}^2)i & \sqrt{2} i\\
(\bar{b}^2 - a^2)i & -(ab+\bar{ab})i & \sqrt{2} i\\
\sqrt{2} i & \sqrt{2} i & 0
\end{pmatrix}
\right\rangle_\lambda\\
&= \frac{-i}{2(|a|^2+3|b|^2)} \left(-3 |a|^2 b^2 \bar{z} - a^2 |b|^2 z + \bar{a}^2 |b|^2 \bar{z} + 3 |a|^2 \bar{b}^2 z - 3 |b|^2 b^2 \bar{z} + |a|^2 \bar{a}^2 \bar{z} - |a|^2 a^2 z + 3 |b|^2 \bar{b}^2 z\right)\\
&= \frac{-i(|a|^2+|b|^2)}{2(|a|^2+3|b|^2)} (-3 b^2 \bar{z} - a^2 z + \bar{a}^2 \bar{z} + 3 \bar{b}^2 z) = \frac{1}{|a|^2 + 3|b|^2} Im(3 b^2 \bar{z} + a^2 z).
\end{align*}
So, $Im(a^2 z + 3 b^2 \bar{z}) = 0$.

Together, these observations imply that $a^2z+3b^2\bar{z} = 0$, so $|a|^2 = 3|b|^2$, hence $|b|^2=1/4$ and $|a|^2=3/4$. Furthermore, given $a$ and $b$, $z$ is unique up to scaling (which also scales $r$). Therefore $B$, if it exists, is unique up to scaling. This proves that each point at which there exists a plane of zero-curvature has a unique such plane.

Furthermore, we have
\[
g = \begin{pmatrix}
\frac{\sqrt{3}}{2} e^{ti} & \frac{1}{2} e^{si} &\\
-\frac{1}{2} e^{-si} & \frac{\sqrt{3}}{2} e^{-ti} &\\
&& 1
\end{pmatrix}
\begin{pmatrix}
w &&\\
&w&\\
&&\bar{w}^2
\end{pmatrix}.
\]
By applying $(\diag(e^{ti},e^{-ti},1),\diag(e^{2ti},e^{-2ti},1))\in SU(2)_\phi$ to $g$, we see that by changing the point in the orbit, we can assume that
\[
g = \begin{pmatrix}
\frac{\sqrt{3}}{2} & \frac{1}{2} e^{s'i} &\\
-\frac{1}{2} e^{-s'i} & \frac{\sqrt{3}}{2}&\\
&& 1
\end{pmatrix}
\begin{pmatrix}
w &&\\
&w&\\
&&\bar{w}^2
\end{pmatrix}.
\]
It is easy to verify that the choice of such representative element is unique up to replacing $w$ by $-w$. Therefore, the set of points with 0-curvature planes is in a one-to-one correspondence to a 2-torus.

We now use a result by Wilking:

\begin{prop}(Wilking \cite{WApos}) Let $M$ be a normal biquotient. Suppose $\sigma\subset T_p M$ is a plane satisfying $\sec(\sigma)=0$. Then the map $\exp:\sigma\to M,v\mapsto \exp(v)$ is a totally geodesic isometric immersion.
\end{prop}

To apply this to a biquotient $G//U$ with a Cheeger deformed metric, observe that $G//U = (G\times\lambda K)//U'$, where $U' = \{((u_l,u_r^{-1}),(k,k)) | (u_l,u_r)\in U, k\in K\}\subset (G\times K)\times (G\times K)$. In this context, Wilking's result tells us that exponentiating a flat plane results in a flat totally geodesic subspace. Let $T$ be the set of all points in $SU(3)//SU(2)_\phi$ with flat planes. If $p\in T$ and $\sigma$ is the unique flat 2-plane at $p$, it follows that near $p$ we have $T=\exp_p\sigma$. In particular, $T$ is smooth and hence diffeomorphic to a 2-torus. Furthermore, it follows that for all $p\in T$, the unique flat plane must be tangent to $T$. This concludes the proof of Theorem \ref{ThmChDef}.
\end{proof}

Theorem \ref{ThmO5} follows immediately from Theorem \ref{ThmChDef}, and in particular, Theorem \ref{ThmChDef} tells us what metric to use for Theorem \ref{ThmO5}. An interesting question is whether the metric in Theorem \ref{ThmChDef} can be further deformed to give a metric of positive curvature. The author has made an attempt to achieve this by doing an additional Cheeger deformation along $SU(2)$ on the left; however, the curvature properties appear to be unchanged.

\begin{proof}[Proof of Corollary \ref{CorAlex}]
Observe that
\[
S^1=\left\{
\begin{pmatrix}
z &&\\
& z &\\
&& \bar{z}^2
\end{pmatrix}
\right\}
\]
acts on the left on $SU(3)$. Furthermore, this $S^1$ commutes with the $SU(2)$ action in the construction of $\calO$. Therefore, we get an Alexandrov space $X^4 = \calO/S^1$.

Furthermore, note that each 0-curvature plane of $\calO^5$ contains a direction (the vector $A$ from before) tangent to the fiber of this action. Therefore, by O'Neil's formula, $X^4$ has positive curvature.

Additionally, note that when $z=-1$, the action is trivial, and corresponds to the action of $-I\in SU(2)_\phi\subset SU(3)\times SU(3)$. Therefore, we conclude that $X^4=SU(3)//U(2)$.
\end{proof}

\section{Eschenburg Spaces and Orbifolds}\label{SecEsch}

\subsection{Seven Dimensional Family}

First introduced in \cite{EHab}, Eschenburg spaces are a rich family of 7-dimensional manifolds (the construction can be generalized to orbifolds as well), that all admit quasi-positive curvature \cite{KThes}, and many of which admit positive curvature. Eschenburg spaces are defined as

\[E^7_{p,q}=SU(3)//S^1_{p,q}\]

where $p,q\in \bbZ^3$, $\sum p_i=\sum q_i$. Furthermore for the action to be free, we need that
\[
(p_1-q_{\sigma(1)},p_2-q_{\sigma(2)})=1\qquad\text{for any }\sigma\in S_3.
\]

More generally, if we allow Eschenburg orbifolds, then the condition is relaxed to $p$ and $q$ not being permutations of each other, in other words, for $\sigma\in S_3$ we have
\[
(p_1-q_{\sigma(1)},p_2-q_{\sigma(2)})\neq 0.
\]

The action of $S^1_{p,q}$ on $SU(3)$ is given by $z\star X = \diag(z^{p_1},z^{p_2},z^{p_3})\cdot X\cdot \diag(\bar{z}^{q_1},\bar{z}^{q_2},\bar{z}^{q_3})$. Eschenburg showed that this space admits a metric of positive curvature when deformed along one of the three block embeddings of $U(2)\subset SU(3)$, iff $q_i\not\in[\min\{p_j\},\max\{p_j\}]$ for each $i$. Kerin further showed that all Eschenburg spaces have quasi-positive curvature, and if
\[
q_1<q_2=p_1<p_2\leq p_3<q_3\;or\;q_1<p_1\leq p_2<p_3=q_2<q_3,
\]
the metric has almost positive curvature. Since all the above results are proven on the Lie algebra level, they hold when we generalize to Eschenburg orbifolds.

Before examining the idea of orbifold fibrations of Eschenburg spaces by Florit and Ziller \cite{FZOrbi}, let us first examine the orbifold structure of Eschenburg orbifolds, since it will be similar to that of the orbifold fibrations.

It is easy to verify that the singular locus of an Eschenburg orbifold $SU(3)//S^1_{p,q}$ consists of some combination of circles denoted $\calC_\sigma$ and lens spaces (possibly including $S^3$ and $S^2\times S^1$) denoted $\calL_{ij}$. Furthermore, each such component is a totally geodesic suborbifold. In this construction we include a minor correction to the work of Florit and Ziller to ensure that $U(2)_{ij}$ and $T^2_\sigma$ are always subsets of $SU(3)$.

We define $\calL_{ij}$ to be the images in $E_{p,q}^7$ of $U(2)_{ij}\subset SU(3)$, defined as
\[
U(2)_{ij}=\left\lbrace
\tau_i
\begin{pmatrix}
A & 0\\
0 & \bar{\det A}
\end{pmatrix}
\tau_j : A\in U(2)\right\rbrace,\qquad 1\leq i,j\leq 3
\]
where $\tau_i\in S_3 \subset O(3)$ with $\tau_1,\tau_2$ interchanging the $3^{rd}$ vector with the $1^{st}$ and $2^{nd}$ respectively, and $\tau_3=-I$.

Furthermore, we define $\calC_\sigma$ for $\sigma\in S_3$ as the projections of $T^2_\sigma$, which are defined as
\[
T^2_\sigma = \sgn(\sigma)\sigma^{-1}\diag(z,w,\bar{zw})
\]
where we view $S_3\subset O(3)$ and $\sgn(\sigma)$ is $1$ if $\sigma$ even and $-1$ if $\sigma$ odd.

From this listing, we can observe that the singular locus has the following structure

\begin{figure}[H]
\begin{tikzpicture}
\fill (90:3cm) circle (2pt) node [above] {$\calC_{Id}$};
\fill (30:3cm) circle (2pt) node [above right] {$\calC_{(12)}$};
\fill (330:3cm) circle (2pt) node [below right] {$\calC_{(132)}$};
\fill (270:3cm) circle (2pt) node [below] {$\calC_{(13)}$};
\fill (210:3cm) circle (2pt) node [below left] {$\calC_{(123)}$};
\fill (150:3cm) circle (2pt) node [above left] {$\calC_{(23)}$};
\draw (90:3cm) -- node [above right] {$\calL_{33}$} (30:3cm) ;
\draw (30:3cm) -- node [right] {$\calL_{21}$} (330:3cm);
\draw (330:3cm) -- node [below right] {$\calL_{13}$} (270:3cm);
\draw (270:3cm) -- node [below left] {$\calL_{31}$} (210:3cm);
\draw (210:3cm) -- node [left] {$\calL_{23}$} (150:3cm);
\draw (150:3cm) -- node [above left] {$\calL_{11}$} (90:3cm);
\draw (90:3cm) -- node [near end,left] {$\calL_{22}$} (270:3cm);
\draw (30:3cm) -- node [near start,below right] {$\calL_{12}$} (210:3cm);
\draw (330:3cm) -- node [near end,above right] {$\calL_{32}$} (150:3cm);
\end{tikzpicture}
\caption{Structure of the singular locus \cite{FZOrbi}}\label{FigSL}
\end{figure}
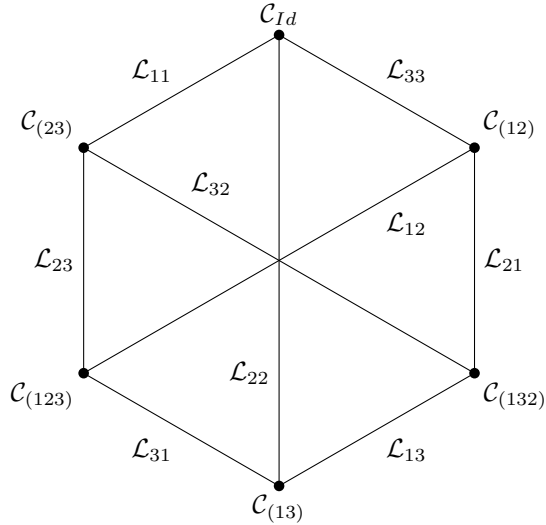

where $\calL_{ij}$ connecting $\calC_{\sigma}$ and $\calC_{\tau}$ means that both $\calC_{\sigma}$ and $\calC_{\tau}$ lie in $\calL_{ij}$. We compute the orbifold groups along $\calC_\sigma$ in Theorem \ref{ThmOG7}, and the orbifold groups along $\calL_{ij}$ are implied by Lemma \ref{LemLC}.

\subsection{Construction of the Six Dimensional Family}

For the construction of the six dimensional family of Exchenburg spaces, Florit and Ziller \cite{FZOrbi} considered fibrations of the form $E_{p,q}^7/S^1_{a,b}$. Given $a,b\in\bbZ^3,\sum a_i=\sum b_i$ we define $S^1_{a,b}$ acting on $SU(3)$ as before, and furthermore, this action induces an action of $S^1_{a,b}$ on $E^7_{p,q}$. In this paper we consider these orbifolds more directly as $\calO_{a,b}^{p,q} = SU(3)//T^2$, where $T^2$ is generated by the two circles $S^1_{p,q}$ and $S^1_{a,b}$.

Florit and Ziller prove that

\begin{thm}
The action of $T^2 = S^1_{a,b}\times S^1_{p,q}$ on $SU(3)$ is almost free iff
\[
(p-q_\sigma)\;and\;(a-b_\sigma)\;are\;linearly\;independent,\;for\;all\;\sigma\in S_3.
\]
The quotient $\calO^{a,b}_{p,q}$ is then an orbifold whose singular locus is the union of at most nine orbifold 2-spheres and six points that are arranged according to the schematic diagram in Figure \ref{FigSL} above.
\end{thm}

\subsection{Equivalence of Actions by $T^2$}

It is clear that $\calO_{p,q}^{a,b}=\calO_{a,b}^{p,q}$. A natural question is what other ways are there to write the same biquotient?

\begin{prop}\label{EqAct}
$\calO_{p,q}^{a,b}=\calO_{p',q'}^{a',b'}$ whenever $a',b',p',q'\in\bbZ^3$ are given as follows:

\begin{enumerate}
\item $a'=b,b'=a,p'=q,q'=p$

\item $a'=\lambda a,b'=\lambda b$, $p'=\mu p, q'=\mu q$ where $\lambda,\mu\in\bbQ\setminus\{0\}$.

\item $a'=(a_1+c,a_2+c,a_3+c),b'=(b_1+c,b_2+c,b_3+c)$, $p'=(p_1+d,p_2+d,p_3+d),q'=(q_1+d,q_2+d,q_3+d)$ where $c,d\in\bbZ$.

\item $a'=\sigma(a),b'=\tau(b),p'=\sigma(p),q'=\tau(q)$, where $\sigma,\tau\in S_3$ act by permutation.

\item
\[
\begin{pmatrix}a'\\p'\end{pmatrix}=A\begin{pmatrix}a\\p\end{pmatrix} \qquad \begin{pmatrix}b'\\q'\end{pmatrix}=A\begin{pmatrix}b\\q\end{pmatrix}
\]

where $A\in GL_2(\bbZ)$.
\end{enumerate}

\end{prop}

\begin{proof}
The first 4 are simply adaptations of the equivalence rules for Eschenburg spaces. The fifth one is simply a reparametrization of the $T^2$ corresponding to a change of basis.\end{proof}

Two important corollaries of this proposition will allow us to only deal with effective actions of $T^2$.

\begin{cor}\label{CorIK}
Given an action of $T^2$ on $SU(3)$ with a finite ineffective kernel, the above operations allow us to write the same quotient as $SU(3)//T^2$ with an effective action.
\end{cor}

\begin{proof}
Let $(z_0,w_0)$ be an element of the ineffective kernel of order $n$. We can choose integers $k,l$ such that $0\leq k,l<n$, $\gcd(k,l,n)=1$, and
\[
z_0=e^{2\pi i k/n},\qquad\qquad w_0 = e^{2\pi i l/n}.
\]
We now consider a different generator. Let $s$ be such that $(k,l)s \equiv 1\pmod{n}$. Consider $(z_1,w_1)=(z_0^s,w_0^s)$, since $(n,s)=1$, we must have $(z_1,w_1)$ and $(z_0,w_0)$ generate the same subgroup, but furthermore, we have
\[
z_1=e^{2\pi i k'/n},\qquad\qquad w_1=e^{2\pi i l'/n},
\]
where $k'=k/(k,l),l'=l/(k,l)$, and so $(k',l')=1$.

Next let $\alpha,\beta$ be integers such that $\alpha l' - \beta k' = 1$. Then, we apply transformation 5 above with
\[
A = \begin{pmatrix}
l' & k'\\
\beta & \alpha
\end{pmatrix}.
\]
Under this transformation, $(z_1,w_1)$ gets changed to $(u_1,v_1)$, where $v_1=1$, and $u_1=e^{2\pi i/n}$. Let the action by $(u,v)$ be denoted as
\[
(u,v)\cdot X = u^{p'}v{a'}X\bar{u}^{q'}\bar{v}^{b'}.
\]
Then, the fact that $(e^{2\pi i/n},1)$ is in the ineffective kernel, means that $p_1\equiv p_2\equiv p_3\equiv q_1\equiv q_2\equiv q_3\pmod{n}$. Apply transformation 3 above with $c=0$ and $d=-p_1$, to get that $p_i\equiv q_i \equiv 0\pmod{n}$. Next applying transformation 2, with $\lambda=1/n$, we kill off this generator of the ineffective kernel.

Repeating this process for all generators of the ineffective kernel guarantees that the action is effective.
\end{proof}

For the next corollary, we consider a special subfamily of the seven dimensional Eschenburg spaces: $E_d^7 = E^7_{(1,1,d),(0,0,d+2)}$ for $d\geq 0$ is a family of cohomogeneity one manifolds, which means that there exists a group $G$ which acts on $E_d^7$ by isometries with $\dim E_d^7/G = 1$. In this case, $G=SO(3)SU(2)$ (see \cite{GWZCoho}). For $d>0$, Eschenburg's construction gives us a metric of positive curvature on $E_d^7$.

In general, the approach in Corollary \ref{CorIK} makes no guarantees that the effective $T^2$ shares a generating circle with the initial $T^2$ action. The following corollary addresses this shortcoming for the particular case when $SU(3)//S^1_{p,q} = E_d^7$.

\begin{cor}\label{CorEAC1}
If $\calO_{p,q}^{a,b} = E_d^7//S^1_{a,b} = SU(3)//T^2$ has ineffective torus action, then we can rewrite it as $E_{d}^7//S^1_{a',b'}$ with an effective torus action.
\end{cor}

\begin{proof}
Let $S^1_{a,b}$ act by $w^{\alpha,\beta,0}$ on the left and $\bar{w}^{\gamma,\delta,\epsilon}$ on the right. Since $\epsilon$ is uniquely determined by the other 4 indecies, we will mostly ignore it.

The innefective kernel has order given by $k = \gcd(\gamma-\delta,\alpha-\beta, \alpha d - \gamma(d-1))$.

In particular, $\alpha\equiv\beta\pmod{k}$ and $\gamma\equiv\delta\pmod{k}$. Our goal now is to make $k|\alpha$ and $k|\gamma$, this would mean that $k$ divides all exponents in the action of $S^1_{a,b}$, and therefore, has ineffective kernel $\bbZ_k$, which we can get rid of.

Let $r \equiv \gamma-\alpha\pmod{k}$. Take $A =\begin{pmatrix}
1&r\\
0&1
\end{pmatrix}\in GL_2(\bbZ)$.

Apply rule 5 from Proposition \ref{EqAct}, and we get $E_d^7//S^1_{a,b} = E_d^7//S^1_{a',b'}$, where $a'=(\alpha+r,\beta+r,rd), b'=(\gamma,\delta,\epsilon)$. Applying rule 3 we get $a''=(\alpha+r(1-d),\beta+r(1-d),0), b''=(\gamma-rd,\delta-rd,\epsilon)$.

Observe that $a''_1-b''_1 = \alpha-\gamma + r \equiv 0 \pmod{k}$. Since this is just a reparametrization of the torus, it still has the same ineffective kernel, so $k | a''_1 d - b''_1(d-1)$, but $a''_1 d - b''_1(d-1) = (a''_1-b''_1)d + b''_1 \equiv b''_1 \pmod{k}$. Therefore, $k|b''_1$, and so $k$ must also divide $a''_i,b''_i$ for all $i$. So, we divide $a'',b''$ by $k$, and get rid of the ineffective kernel.
\end{proof}

\subsection{Orbifold Groups at $\calC_\sigma$ and $\calL_{ij}$}

We assume from now on that the action of $T^2$ on $SU(3)$ is effective. The following lemma is essential to understanding the orbifold group at points in $\calL_{ij}$ in terms of the orbifold groups on the $\calC_{\sigma}$'s it connects.

\begin{lem}\label{LemLC}
Let $\calL_{ij}$ connect $\calC_{\sigma}$ and $\calC_{\tau}$, then $(z,w)$ acts trivially on $U(2)_{ij}$ iff $(z,w)$ acts trivially on $T^2_\sigma$ and $T^2_\tau$.
\end{lem}

\begin{proof}
One direction is trivial. If $(z,w)$ acts trivially on $U(2)_{ij}$, then it acts trivially on every subset, in particular the two torii.

By conjugation, we may assume without loss of generality that $\sigma=Id, \tau=(12)$, so $\calL_{ij} = \calL_{33}$, and $U(2)_{ij}$ is the standard embedding of $U(2)$ into $SU(3)$. Now assume that $(z,w)$ acts trivially on $T^2_{Id}$ and $T^2_{(12)}$. Since $I\in T^2_{Id}$, and $(z,w)\star X = \diag(u_1,u_2,u_3)\cdot X\cdot \diag(v_1,v_2,v_3)^{-1}$, if follows that $v_i=u_i$. Now observe that $A_{(12)} = \begin{pmatrix}
0 & -1 & 0\\
-1 & 0 & 0\\
0 & 0 & -1
\end{pmatrix}\in T^2_{(12)}$. Hence $(z,w)\star A_{(12)} = A_{(12)}$, implies that $u_1\bar{u_2}=1$, so $u_1=u_2$. Therefore, the action of $(z,w)$ becomes $(z,w)\star X =  \diag(u_1,u_1,u_3)X\diag(u_1,u_1,u_3)^{-1}$. Thus, $(z,w)$ fixes all of $U(2)_{33}$ as well.
\end{proof}

The following results all assume that the action of $S^1$ or $T^2$ is effective.

\begin{thm}\label{ThmOG7}
Let $\Gamma_\sigma^{p,q}$ denote the orbifold group of $E_{p,q}^7$ along $C_\sigma$. Then $\Gamma_\sigma^{p,q}$ is a cyclic group of order $\gcd(p_1-q_{\sigma(1)},p_2-q_{\sigma(2)})$.
\end{thm}

\begin{proof}
Let $z\in S^1$ be an element that fixes $T_\sigma^2$, then $z^{p_i-q_{\sigma(i)}} = 1$ for $i=1,2,3$. In particular, if $r = \gcd(p_1-q_{\sigma(1)},p_2-q_{\sigma(2)})$, then $z^r=1$, and in fact, any $z$ satisfying $z^r=1$ fixes $T_\sigma^2$. Therefore, $\Gamma_\sigma^{p,q}=\bbZ_r$.
\end{proof}

\begin{thm}\label{ThmOG}
For $\calO^{a,b}_{p,q}$, the orbifold group at $\calC_\sigma$ denoted by $\Gamma_\sigma$ has order 
\[
N_\sigma = \left|(p_1-q_{\sigma(1)}) (a_2-b_{\sigma(2)}) -
(a_1-b_{\sigma(1)})(p_2-q_{\sigma(2)})\right|.
\]

Let $r_\sigma = \gcd(|\Gamma_\sigma^{a,b}|,|\Gamma_\sigma^{p,q}|)$, then $\Gamma_\sigma = \bbZ_{r_\sigma}\oplus\bbZ_{N_\sigma/r_\sigma}$. In particular, $\Gamma_\sigma$ is non-cyclic iff the orders of $\Gamma_\sigma^{a,b}$ and $\Gamma_\sigma^{p,q}$ are not relatively prime.
\end{thm}

\begin{rmk}
In particular, it follows that if $\gcd(|\Gamma_\sigma^{a,b}|,|\Gamma_\sigma^{p,q}|)>1$, then $\Gamma_\sigma$ is non-cyclic, and so there is no Eschenburg 7-manifold $E^7_{u,v}$ such that $\calO^{a,b}_{p,q}=E^7_{u,v}//S^1$.

For example, consider $p=(0,0,1),q=(2,4,-5),a=(0,1,1),b=(2,3,-3)$, then $\Gamma_{id} = \bbZ_2\oplus\bbZ_2$ is non-cyclic, and so $\calO^{a,b}_{p,q}$ can not be written as $E^7_{u,v}//S^1$ where $E^7_{u,v}$ is an Eschenburg manifold.
\end{rmk}

\begin{proof}
The order of $\Gamma_\sigma$ follows immediately from Proposition 3.7 in \cite{FZOrbi}. All that remains to show is the assertion about its group structure.

Since $\Gamma_\sigma\subset T^2$, we conclude that it is either cyclic or a direct sum of two cyclic groups. It is clearly non-cyclic iff it has a subgroup of the form $\bbZ_n\oplus\bbZ_n$. We will show that this occurs iff $n$ divides both $|\Gamma_\sigma^{a,b}|$ and $|\Gamma_\sigma^{p,q}|$.

If $n$ divides both $|\Gamma_\sigma^{a,b}|$ and $|\Gamma_\sigma^{p,q}|$, then clearly any element of the form $(z,w) = (e^{2k\pi i/n},e^{2l\pi i/n})$ fixes $T^2_\sigma$, and so $\bbZ_n\oplus\bbZ_n\subset\Gamma_\sigma$. Conversely, suppose that $\bbZ_n\oplus\bbZ_n\subset\Gamma_\sigma$, then we must have at least $n^2$ elements $(z,w)\in T^2$ that satisfy $(z^n,w^n)=(1,1)$. However, there are precisely $n^2$ such elements, which implies that all of them act trivially on $T^2_\sigma$. In particular, $(e^{2\pi i/n},1)$ and $(1,e^{2\pi i/n})$ fix $T^2_\sigma$, so $e^{2\pi i/n}$ is both in $\Gamma_\sigma^{a,b}$ and in $\Gamma_\sigma^{p,q}$. Therefore, $n$ divides the order of both groups. This completes the proof.
\end{proof}

\begin{cor} \label{CorOrbiC}
Let $S^1_{a,b}$ act on $E_d^7$, then the orders of the orbifold groups are given by

\begin{table}[H]
\centering
\begin{tabular}{|c|c|}
\hline
$\sigma$ & $N_\sigma$ \\\hline
$id$ & $|(\alpha-\beta)-(\gamma-\delta)|$\\\hline
$(12)$ & $|(\alpha-\beta)+(\gamma-\delta)|$\\\hline
$(13)$ & $|\gamma+d(\beta-\delta)|$\\\hline
$(123)$ & $|\gamma+d(\alpha-\delta)|$\\\hline
$(132)$ & $|\delta+d(\beta-\gamma)|$\\\hline
$(23)$ & $|\delta+d(\alpha-\gamma)|$\\\hline
\end{tabular}
\caption{$\calC_\sigma$ singularities in $E^7_d/S^1_{a,b}$}
\end{table}

Where $a=(\alpha,\beta,0)$ and $b=(\gamma,\delta,\alpha+\beta-\gamma-\delta)$.
\end{cor}

This corollary follows from Theorem \ref{ThmOG}. Alternatively, the same results can be obtained from Proposition 3.7 in \cite{FZOrbi}.

\begin{thm} \label{ThmOrbiL}
Let $S^1_{a,b}$ act on $E_d^7$ with $a=(\alpha,\beta,0)$ and $b=(\gamma,\delta, \alpha+\beta-\gamma-\delta)$. Then, the orbifold groups $\Gamma_{ij}$ of $\calL_{ij}$ is cyclic with order $N_{ij}$ given by the following table:

\begin{table}[H]
\centering
\begin{tabular}{|c|c|}\hline
$i,j$ & $N_{ij}$ \\\hline
1,1 & $((\alpha-\beta)-(\gamma-\delta),\delta+d(\alpha-\gamma))$\\\hline
1,2 & $((\alpha-\beta)+(\gamma-\delta),\gamma+d(\alpha-\delta))$\\\hline
1,3 & $(\delta-\gamma, \delta+d(\beta-\gamma))$\\\hline
2,1 & $((\alpha-\beta)+(\gamma-\delta),\delta+d(\beta-\gamma))$\\\hline
2,2 & $((\alpha-\beta)-(\gamma-\delta),\gamma+d(\beta-\delta))$\\\hline
2,3 & $(\delta-\gamma, \delta+d(\alpha-\gamma))$\\\hline
3,1 & $(\alpha-\beta, \gamma+d(\alpha-\delta))$\\\hline
3,2 & $(\alpha-\beta, \delta+d(\alpha-\gamma))$\\\hline
3,3 & $(\alpha-\beta,\gamma-\delta)$\\\hline
\end{tabular}
\caption{$\calL_{ij}$ singularities in $E^7_d/S^1_{a,b}$}
\end{table}
\end{thm}

\begin{proof}
This is an immediate consequence of part (c) of Proposition 3.7 of \cite{FZOrbi}.
\end{proof}

\subsection{Corrections to Theorem C \cite{FZOrbi}}

In this section we examine Theorem C of \cite{FZOrbi} and provide both corrections and improvements to it.

\begin{thm}\label{ThmFix}
Let $E_d$ be a cohomogeneity one Eschenburg manifold, $d\geq 3$, equipped with a positively curved Eschenburg metric. Then:

\begin{enumerate}
\item[i)] If $S^1$ acts on $E_d^7$ by isometries, then there are at minimum 3 singular points, in particular, if exactly two $C_\sigma$'s are singular, then the $\calL_{ij}$ connecting them is also singular.
\end{enumerate}
In the following particular examples the singular locus of the isometric circle action $S^1_{a,b}$ on $E_d$ consists of:

\begin{enumerate}
\item[ii)] A smooth totally geodesic 2-sphere with orbifold group $\bbZ_{d+1}$ if $a=(0,-1,1)$ and $b=(0,0,0)$;

\item[iii)] When $a=(0,1,1)$and $b=(2,0,0)$, the singular locus consists of four point with orbifold groups $\bbZ_3,\bbZ_{d+1},\bbZ_{d+1},\bbZ_{2d+1}$, and the following orbifold groups on spheres:

If $3|(d+1)$, then the first 2 points are connected by a totally geodesic 2-sphere with orbifold group $\bbZ_3$.

If $3|(d-1)$, then the first and the fourth points are connected by a totally geodesic 2-sphere with orbifold group $\bbZ_3$.

If $2|(d+1)$, then the second and the third points are connected by a totally geodesic 2-sphere with orbifold group $\bbZ_2$.

\item[iv)] A smooth totally geodesic 2-sphere with orbifold group $\bbZ_{d-1}$ if $a=(0,1,1)$ and $b=(0,0,2)$.

\item[v)] Three isolated singular points with orbifold groups $\bbZ_{2d-3}, \bbZ_{d^2-d-1},\bbZ_{d^2-d-1}$ if $a=(0,d-1,0)$ and $b=(1,d-1,-1)$.
\end{enumerate}
\end{thm}

\begin{proof}

Parts 2-5 are direct application of Theorem \ref{ThmOrbiL} and Corollary \ref{CorOrbiC}. Part 1 deserves a special mention:

Assume that $S^1_{a,b}$ acts on $E_d^7$ in such a way that at most two of $N_{\sigma}$'s are not 1 (i.e. at most two $\calC_\sigma$'s are singular). We start with a lemma:

\begin{lem}
If $\alpha=\beta$ or $\gamma=\delta$, then the singular locus of $E_d^7//S^1_{a,b}$ consists of smooth totally geodesic 2-spheres. Where $a=(\alpha,\beta,0)$ and $b=(\gamma,\delta,\alpha+\beta-\gamma-\delta)$.
\end{lem}

\begin{proof}
Suppose $\alpha=\beta$, and $(z,w)$ fixes $T^2_\sigma$, then look at $U(2)_{33}T^2_\sigma$. $\alpha=\beta$ implies that the matrix acting on the left is of the form $\diag(u,u,v)$, which commutes with $U(2)_{33}$. Therefore, $(z,w)$ fixes $U(2)_{33}T_\sigma^2$.

If $\gamma=\delta$, apply the same argument to $T^2_\sigma U(2)_{33}$.

If $\calL_{ij}$ is not of the form above and is singular, then so is either $U(2)_{33}U(2)_{ij}$ or $U(2)_{ij}U(2)_{33}$, which contradicts the orbifold structure of the singular locus.
\end{proof}

We now split the $N_\sigma$'s into 3 pairs using Corollary \ref{CorOrbiC}:

\begin{itemize}
\item If $N_{id}=N_{(12)}=1$, then either $\alpha=\beta$ or $\gamma=\delta$;

\item If $N_{(13)}=N_{(123)}=1$, then $\alpha=\beta$ (since $d\geq 3$);

\item If $N_{(23)}=N_{(132)}=1$, then $\alpha=\beta$ (since $d\geq 3$).
\end{itemize}

Since at most two $N_\sigma$'s are not 1, we see that at least one of the above cases must occur. Therefore, by the lemma above, we see that $E_d^7//S^1_{a,b}$ has singular locus consisting of smooth totally geodesic 2-spheres.

It remains to show that there is no free action of $S^1_{a,b}$ on $E_d^7$.

Suppose that the action is free, then $N_{(13)}=N_{(123)}=1$, so $\alpha=\beta$. Furthermore, $N_{(13)}=N_{(132)}=1$, so we get $(d+1)(\gamma-\delta)$ is either -2, 0 or 2, but $d\geq 3$, so we must have $\gamma=\delta$, which implies $N_{id}=0$, so the quotient $E_d^7//S^1_{a,b}$ is not an orbifold.
\end{proof}

As a corollary, we get

\begin{rmk}
Parts 1 and 4 of Theorem C from \cite{FZOrbi} hold, but parts 2 and 3 are false.
\end{rmk}

\subsection{Curvature}

The goal of this section is to prove the following theorem that imposes restrictions on when $SU(3)//T^2$ admits a metric of positive curvature:

\begin{thm}\label{Thm6DPosC}
Given an orbifold $\calO_{p,q}^{a,b}$ which has positive curvature induced by a Cheeger deformation along $U(2)$, there exists $E^7_{u,v}$ (either a manifold or an orbifold) such that $\calO_{p,q}^{a,b} = E^7_{u,v}//S^1$ and $E^7_{u,v}$ has positive curvature induced by Cheeger deformation along the same $U(2)$.

Equivalently, there exist $\lambda,\mu\in\bbZ$ relatively prime such that $E^7_{\lambda p + \mu a,\lambda q + \mu b}$ is positively curved.
\end{thm}

The following is an example of why this theorem is non-trivial.

\begin{ex}
Consider the Eschenburg orbifold $\calO_{p,q}^{a,b}$ given by $a = (-2,0,2), b = (-3,1,2), p = (-4,0,2)$ and $q = (-5,3,0)$. From the work of Eschenburg, it follows that deforming by $U(2)$ does not result in a metric of positive curvature on either $E^7_{p,q}$ or $E^7_{a,b}$. However, we can re-write this orbifold as $\calO^{p',q'}_{a',b'}$, where $a'=a,b'=b$ and $p'=2a-p = (0,0,2),q'=2b-q = (-1,-1,4)$, this is the same orbifold according to Proposition \ref{EqAct}. Additionally now $E^7_{p',q'} = E^7_{2}$ admits positive curvature.
\end{ex}

Before proving this theorem, we state a lemma that limits the parametrizations we need to consider:

\begin{lem}\label{LemRepar}
Let $\calO_{p,q}^{a,b}$ be a 6-dimensional Eschenburg space, then there exist a reparametrization $p',q',a',b'$ (not necessarily effective) satisfying one of the following:

\begin{itemize}
\item $p'_1=p'_2=p'_3=0$, or

\item $p'_1=p'_3=a'_1=0,p'_2=a'_2=a'_3=n$ for some $n\in\bbZ^+$.
\end{itemize}
\end{lem}

\begin{proof}
We begin with an intermediate reparametrization $p^0,q^0,a^0,b^0$ where $p_1^0=a_1^0=0$. Now consider $\Delta = a_2^0 p_3^0 - a_3^0 p_2^0$.

Suppose that $\Delta=0$, then we want to show that we have the first scenario. Additionally suppose $p_2^0\neq 0$ (if $p_2^0=0,p_3^0\neq 0$ just apply $(12)\in S_3$ to both $p$ and $a$ to get $p_2^0=0$, and if both are 0, we already have case 1).

If $a_3^0\neq 0$, then $a_2^0\neq0,p_3^0\neq 0$. This implies that $p_2^0 = t a_2^0,p_3^0=t a_3^0$ for some $t\in\bbQ\setminus\{0\}$, so there exist $m,n\in\bbZ$ relatively prime so that $m p_i^0 + n a_i^0=0$ for all $i$ ($t=m/n$). Now consider $m',n'\in\bbZ$ such that $m' m - n' n = 1$, define $p' = m p^0 + n a^0,a' = n' p^0 + m' a^0$ (analogously for $q',b'$), then $p'_i=0$ for all $i$ and we have case 1.

If $a_3^0=0$, then $a_2^0=0$ or $p_3^0=0,a_2^0\neq 0$, if the former, then let $p'=a^0,a'=p^0$ giving us case 1. In the latter case, let $m,n$ be such that $m p_2^0 + n a_2^0=0$. Now consider $m',n'\in\bbZ$ such that $m' m - n' n = 1$, define $p' = m p^0 + n a^0,a' = n' p^0 + m' a^0$ (analogously for $q',b'$), then $p'_i=0$ for all $i$ and we have case 1.

Next suppose that $\Delta \neq 0$, our goal is to show that we have the second scenario. Consider $\tilde{p} = -a_3^0 p + p_3^0 a,\tilde{a} = -a_2^0 p + p_2^0 a$ (similarly for $q,b$). Then $\tilde{p}_1=\tilde{p}_3=\tilde{a}_1=\tilde{a}_2=0$, let $n=lcm(\tilde{p_2},\tilde{a_3})$ and $k = n/\tilde{p_2},l=n/\tilde{a_3}$. Then define $p' = k \tilde{p}, a' = k\tilde{p} + l\tilde{a}$. This gives us $p' = (0,n,0),a'=(0,n,n)$.
\end{proof}

We also prove the explicit conditions in terms of $a,b,p,q$ for when the orbifold $\calO_{p,q}^{a,b}$ has positive sectional curvature:

\begin{prop}\label{PropCurv}
Let $\calO_{p,q}^{a,b}$ be as above, and the metric being one given by Cheeger deformation along $U(2)$. Then, $\calO_{p,q}^{a,b}$ is positively curved iff for each $t\in[0,1]$ and each triple $(\eta_1,\eta_2,\eta_3)$ satisfying $\eta_i\geq0$ and $\sum\eta_i=1$ we have both
\setcounter{equation}{0}
\begin{align}
(1-t)b_1 + tb_2 &\neq \sum\eta_i a_i &&OR&(1-t)q_1+tq_2&\neq\sum\eta_i p_i
\end{align}
and
\begin{align}
b_3&\neq\sum\eta_i a_i &&OR&q_3&\neq\sum \eta_i p_i
\end{align}
\end{prop}

\begin{rmk}
We point out that which half of each condition is satisfied can in general depend on the choice of $\eta_i's$ and $t$.
\end{rmk}

\begin{proof}
This is a fairly straightforward application of Eschenburg's original results on the curvature of Eschenburg spaces. In particular, we know that $\sec\sigma=0$ iff one of the following vectors is in $\sigma$
\[
Y_3 = \begin{pmatrix}
i&&\\
&i&\\
&&-2i
\end{pmatrix} \qquad\qquad
\Ad(k)Y_1=k \begin{pmatrix}
-2i&&\\
&i&\\
&&i
\end{pmatrix}k^{-1}\qquad(k\in U(2)).
\]

Condition 2 corresponds to verifying that $Y_3$ is not horizontal, and condition 1 corresponds to verifying that $\Ad(k)Y_1$ is not horizontal.

Let $V_{a,b}(X),V_{p,q}(X)$ be the vectors tangent to the action of $S^1_{a,b}$ and $S^1_{p,q}$ respectively at the point $X\in SU(3)$. It is easy to see that $V_{a,b}(X)=X^{-1}\cdot A\cdot X - B$, where $A = \diag(a_1 i, a_2 i, a_3 i)$ and $B=\diag(b_1 i, b_2 i, b_3 i)$ (similarly for $V_{p,q}(X)$). Let
\[
X = \begin{pmatrix}
x_{11} & x_{12} & x_{13}\\
x_{21} & x_{22} & x_{23}\\
x_{31} & x_{32} & x_{33}
\end{pmatrix}
\]

To verify that $Y_3$ is not horizontal, we need only consider the diagonal entries of $V_{a,b}(X),V_{p,q}(X)$, which are of the form
\[
\left(\displaystyle\sum_{j=1}^3 |x_{j1}|^2a_j-b_1\right)i,\qquad \left(\displaystyle\sum_{j=1}^3 |x_{j2}|^2a_j-b_2\right)i,\qquad \left(\displaystyle\sum_{j=1}^3 |x_{j3}|^2a_j-b_3\right)i
\]

$Y_3$ is orthogonal to $V_{a,b}(X)$ iff $b_3=\sum|x_{j3}|^2 a_j$. Similar condition holds for $Y_3$ being orthogonal to $V_{p,q}(X)$. Therefore, $Y_3$ is not horizontal iff condition (2) holds.

We will approach the question of whether $\Ad(k)Y_1$ is horizontal differently. First observe that $\langle\Ad(X)^{-1}A-B,\Ad(k)Y_1\rangle = \langle\Ad(Xk)^{-1}A-\Ad(k)^{-1}B,Y_1\rangle$. Let
\[
Xk = \begin{pmatrix}
x_{11} & x_{12} & x_{13}\\
x_{21} & x_{22} & x_{23}\\
x_{31} & x_{32} & x_{33}
\end{pmatrix},\qquad\qquad
k = \begin{pmatrix}
c\alpha & c\beta&\\
-c\bar{\beta} & c\bar{\alpha}&\\
&&\bar{c}^2
\end{pmatrix}.
\]

Then, the diagonal entries of $\Ad(Xk)^{-1}A-\Ad(k)^{-1}B$ are
\begin{align*}
&\left[\sum_{j=1}^3 |x_{j1}|^2 a_j - \left(b_1|\alpha|^2+b_2|\beta|^2\right)\right]i,&&
\left[\sum_{j=1}^3 |x_{j2}|^2 a_j - \left(b_1|\beta|^2+b_2|\alpha|^2\right)\right]i,\\
&\left[\sum_{j=1}^3 |x_{j3}|^2 a_j - b_3\right]i
\end{align*}

Taking the inner product with $Y_1$, we get:
\[
3\left[|\alpha|^2 b_1 + |\beta|^2 b_2 - \sum_{j=1}^3 |x_{j1}|^2 a_j\right] + \sum_{j=1}^3 a_j - \sum_{j=1}^3 b_j.
\]

Letting $\eta_j=|x_{j1}|^2,t=|\beta|^2$, we get that $\Ad(k)Y$ is orthogonal to $V_{a,b}^X$ iff
\[
(1-t)b_1+tb_2 = \sum_{j=1}^3 \eta_ja_j
\]

Obtaining a similar formula for $V_{p,q}^X$, we conclude that $\Ad(k)Y$ is not horizontal iff condition (1) holds.
\end{proof}

With these results established, we prove the main result:

\begin{proof}[Proof of Theorem \ref{Thm6DPosC}]
This proof is organized according to the cases given in Lemma \ref{LemRepar}.

\textbf{Case 1:} $p_1,p_2,p_3=0$.

First suppose that $q_1,q_2$ are both positive or both negative, then $E^7_{p,q}$ has positive curvature. Therefore, we will assume, without loss of generality, that $q_1\leq 0 \leq q_2$.

Next suppose that $b_2 - b_1 = q_2 - q_1$. If $b_1 - q_1\in[\min a_i, \max a_i]$, then consider $\alpha\in[0,1], \eta_i\geq 0, \sum\eta_i = 1$ such that $b_1 - q_1 = \sum \eta_i a_i$, and $\alpha q_1 + (1-\alpha)q_2 = 0 = \sum \eta_i p_i$, then $\alpha b_1 + (1-\alpha) b_2 = \alpha q_1 + (1-\alpha) q_2 + \sum\eta_i a_i = \sum\eta_i a_i$, which violates the positivity of sectional curvature for $\calO_{p,q}^{a,b}$. Now suppose that $b_1 - q_1\not\in[\min a_i, \max a_i]$, then there exists $n\in\bbZ^+$ satisfying $q_1 + t(b_1-q_1), q_2 + t(b_2-q_2)$ both $<\min t a_i$ or both $>\max t a_i$. Consider $u = (1-t)p + ta$ and $v = (1-t)q + tb$, then $E^7_{u,v}$ has positive sectional curvature.

Now suppose that $b_2 - b_1 \neq q_2 - q_1$. Then there exist $m,n\in\bbZ$ such that $n q_1 + m(b_1-q_1) = n q_2 + m(b_2 - q_2)$. If $n q_1 + m(b_1-q_1)\not\in[\min m a_i,\max m a_i]$, then take $u = (n-m)p + m a$ and $v = (n-m)q + m b$ to get $E^7_{u,v}$ with positive curvature. Otherwise we have $n q_1 + m(b_1-q_1) \in [\min m a_i, \max m a_i]$, then let $\eta_i$ be such that $n q_1 + m (b_1 - q_1) = m\sum\eta_i a_i$ and $\alpha$ such that $\alpha q_1 + (1-\alpha)q_2 = 0 = \sum\eta_i p_i$. This implies that $m \left(\alpha b_1 + (1-\alpha) b_2\right) = m\sum\eta_i a_i$. So either $\alpha b_1 + (1-\alpha) b_2 = \sum\eta_i a_i$, in which case we don't have $\sec>0$ on $\calO_{p,q}^{a,b}$ or $m=0$ and so $q_1=q_2=0$, however, this implies $q_3 = 0$ as well, and so we have a degeneracy.

\textbf{Case 2:} $p_1 = p_3 = a_1 = 0, p_2 = a_2 = a_3 = n>0$.

\textit{Subcase 2a:} $b_1 - q_1 = b_2 - q_2 = k$.

If $k < 0$ or $k > n$, there exists $t\in\bbZ^+$ such that $q_1 + t(b_1-q_1)$ an $q_2 + t(b_2-q_2)$ are both $<0$ or $>tn$ respectively. Then take $u = (1-t)p + ta,v=(1-t)q + tb$ and we get that $E^7_{u,v}$ has $\sec>0$.

Next, without loss of generality, assume $q_1\leq q_2$. If $[q_1,q_2] \cap [0,n] = \emptyset$, then $E^7_{p,q}$ has positive curvature, otherwise let $m = \min ([q_1,q_2]\cap [0,n]) = \max\{q_1,0\}$. Suppose that $b_1 - q_1 \in [0, n-m]\subset [0,n]$, then pick $\eta_i$ such that $\eta_2 = m/n, n\eta_3 = b_1 - q_1$ and pick $\alpha$ such that $\alpha q_1 + (1-\alpha) q_2 = m$. This implies that $\sum\eta_i p_i = m = \alpha q_1 + (1-\alpha) q_2$ and $\sum \eta_i a_i = m + (b_1 - q_1) = \alpha b_1 + (1-\alpha) b_2$, so $\calO_{p,q}^{a,b}$ does not have $\sec>0$. Next observe that if $q_1\leq 0 \leq q_2$, then $m=0$ and so every possible value of $b_1-q_1$ has been handled. Finally, suppose that $q_1\in(0,n]$, then $m=q_1$, so we need only consider $b_1-q_1\in(n-m,n]$, but this implies that $n< b_1 \leq b_2$, so $E^7_{a,b}$ has positive curvature.

\textit{Subcase 2b:} Without loss of generality, $b_1 - q_1 < b_2 - q_2$. This implies that there exist $k,l\in\bbZ$ such that $k q_1 + l(b_1 - q_1) = k q_2 + l(b_2-q_2)$, $k>0$. Additionally, let $t_0 = kq_1 + l(b_1-q_1)$.

Suppose that $t_0 \not\in[\min\{k p_i + l(a_i-p_i)\},\max\{k p_i + l(a_i-p_i)\}]$. Then, let $u=(k-l)p + l a, v = (k-l)q + l b$ to get $\sec>0$ on $E^7_{u,v}$.

\begin{lem}\label{LemSlope}
Consider all the possible values $\eta_i$ such that $k q_1 + l(b_1-q_1) = \sum\eta_i[k p_i + l(a_i-p_i)]$, and let $\eta_3^m,\eta_3^M$ denote the smallest and largest values of $\eta_3$ respectively.

If $[b_1-q_1,b_2-q_2] \cap [n\eta_3^m,n\eta_3^M]\neq\emptyset$, then $\calO_{p,q}^{a,b}$ does not have $\sec>0$.
\end{lem}

\begin{proof}
Pick $\alpha$ such $\alpha(b_1-q_1) + (1-\alpha)(b_2-q_2) = n\eta_3'$, then $\alpha k q_1 + (1-\alpha) k q_2 = \sum k\eta_i' p_i$, so $\alpha q_1 + (1-\alpha)q_2 = \sum\eta_i' p_i$. We also get $\alpha b_1 + (1-\alpha) b_2 = \sum\eta_i' a_i$, so we do not have $\sec>0$.
\end{proof}

The table below demonstrates the possible relations between $k,l,t_0$ and the corresponding $\eta_3^m,\eta_3^M$:

\begin{table}[H]
\centering
\begin{tabular}{|c|c|c|}
\hline
Relation & $n \eta_3^m$ & $n \eta_3^M$\\\hline
$0 < k < t_0/n \leq l$ & $\frac{t_0-kn}{l-k}$ & $\frac{t_0}{l}$\\\hline
$0 \leq t_0/n \leq k$, $0<l$ & $0$ & $\frac{t_0}{l}$\\\hline
$0 \leq t_0/n \leq k$, $l\leq 0$ & $0$ & $\frac{t_0-kn}{l-k}$\\\hline
$l \leq t_0/n \leq 0 < k$, $l<0$ & $\frac{t_0}{l}$ & $\frac{t_0-kn}{l-k}$\\\hline
\end{tabular}
\caption{Bounds on $\eta_3$}
\end{table}

If $b_i-q_i > t_0/l$ and $l\geq 0$, then $q_i < 0 = \min p_i$, and so $E^7_{p,q}$ has $\sec>0$.

If $b_i-q_i < t_0/l$ and $l < 0$, then $q_i < 0 = \min p_i$, and so $E^7_{p,q}$ has $\sec>0$.

If $b_i-q_i < 0$, then if we take $N>0$ sufficiently large, we get $q_i + N(b_i-q_i)<0$ for $i=1,2$, and so $E^7_{u,v}$ with $u=(1-N) p + N a, v = (1-N) q + N b$ has $\sec>0$.

If $b_i-q_i < (t_0-kn)/(l-k)$ and $l>k>0$, then $b_i > n = \max a_i$, and so $E^7_{a,b}$ has $\sec>0$.

If $b_i-q_i > (t_0-kn)/(l-k)$ and $l\leq0<k$, then $b_i > n = \max a_i$, and so $E^7_{a,b}$ has $\sec>0$.
\end{proof}

\begin{rmk}
We also note that if $\calO^{a,b}_{p,q}$ has positive sectional curvature, then there exist $a',b',p',q'$ such that $\calO^{a',b'}_{p',q'}=\calO^{a,b}_{p,q}$ and $E^7_{a',b'},E^7_{p',q'}$ both have positive sectional curvature. This is achieved by finding $p',q'$ in accordance with the theorem, and taking a circle $S^1_{a',b'}\subset T^2$ sufficiently close to $S^1_{p',q'}$.
\end{rmk}

\bibliographystyle{amsalpha}
\bibliography{References}

\end{document}